\documentclass{amsart}[12pt]

\usepackage{amsmath}
\usepackage{amsfonts}
\usepackage{amssymb}
\DeclareMathOperator*{\dist}{dist}

\newtheorem{theorem}{Theorem}
\newtheorem{lemma}{Lemma}

\newtheorem{corollary}{Corollary}

\theoremstyle{definition}

\def\B{\mathcal{B}}
\def\D{\mathbb{D}}

\newcommand{\bsk}{{\bigskip}}

\def\bege{\begin{equation}} \def\ende{\end{equation}}
\def\bsk{\bigskip}

\def\begr{\begin{eqnarray}} \def\endr{\end{eqnarray}}

\def\bege{\begin{equation}} \def\ende{\end{equation}}
\def\begr{\begin{eqnarray}} \def\endr{\end{eqnarray}}
\def\bnum{\begin{enumerate}} \def\enum{\end{enumerate}}

\begin{document}
\title[Embedding theorem for Dirichlet type spaces ]%
{Embedding theorem for Dirichlet type spaces}

\author{  Junming Liu,  Cheng Yuan and Songxiao Li$^*$ }

\address{Junming Liu  \\ School of Applied Mathematics, Guangdong University of Technology, Guangzhou, Guangdong 510006, P.~R.~China.}
 \email{jmliu@gdut.edu.cn }

\address{Cheng Yuan \\ School of Applied Mathematics, Guangdong University of Technology, Guangzhou, Guangdong 510006, P.~R.~China.}
 \email{yuancheng1984@163.com}

\address{Songxiao Li\\ Institute of Fundamental and Frontier Sciences, University of Electronic Science and Technology of China,
610054, Chengdu, Sichuan, P.R. China\newline
Institute of Systems Engineering, Macau University of Science and Technology, Avenida Wai Long, Taipa, Macau. } \email{jyulsx@163.com}

\begin{abstract}
Using Carleson measure theorem of weighted Bergman spaces, we provide a complete characterization of embedding theorem for Dirichlet type spaces. As an application, we study the Volterra integral operator and multipliers for Dirichlet type spaces.
\end{abstract}

\thanks{This work was supported by NNSF of China (Grant No. 11801094 and No.11720101003).}
 \thanks{* Corresponding author.}


\maketitle

\section{Introduction}
Let $\mathbb{D}$ be the unit disc of complex plane $\mathbb{C}$, and $\partial \mathbb{D}$ be the unit circle. Denote $H(\mathbb{D})$ the space of all analytic
 functions in $\mathbb{D}$. The Bloch space $\mathcal{B}$ consists of all functions $f\in H(\mathbb{D})$ with
$$\|f\|_{\mathcal{B}}=\sup_{z\in \mathbb{D}}(1-|z|^{2})|f'(z)|<\infty.$$ The Bloch space $\mathcal{B}$ is a M\"{o}bius invariant function space under the semi-norm
$\|\cdot\|_{\mathcal{B}}$  and becomes a Banach space with the norm
$\|f\|=|f(0)|+\|f\|_{\mathcal{B}}.$
The little Bloch space $\mathcal{B}_{0}$ is the space of $f\in \mathcal{B}$ with
$$\lim_{|z|\rightarrow 1^{-}}(1-|z|^{2})|f'(z)|=0.$$
The space of bounded analytic function $H^{\infty}$ consists of functions  $f\in H(\mathbb{D})$ satisfying
$\|f\|_{\infty}=\sup_{z\in \mathbb{D}}|f(z)|<\infty.$

  Let $0<p<\infty$.  The Dirichlet type space $\mathcal{D}^{p}_{p-1}$ consist of all functions $f\in H(\mathbb{D})$ for which
 $$\|f\|_{\mathcal{D}^{p}_{p-1}}=\Big(|f(0)|^{p}+\int_{\mathbb{D}}|f'(z)|^{p}(1-|z|^{2})^{p-1}dA(z)\Big)^{1/p}<\infty,$$
 where $dA(z)=\frac{1}{\pi}dxdy$ is the normalized area measure on $\mathbb{D}$. For $p=2$, $\mathcal{D}^{p}_{p-1}$ is the classical Hardy space $H^{2}(\mathbb{D})$.

 Let $0<p<\infty$, $q>-2$ and $0<s<\infty$. The $F(p,q,s)$ space consist of those functions  $f\in H(\mathbb{D})$ with
 $$\|f\|_{F(p,q,s)}^{p}=|f(0)|^{p}+\sup_{a\in \mathbb{D}}\int_{\mathbb{D}}|f'(z)|^{p}(1-|z|^{2})^{q}(1-|\sigma_{a}(z)|^{2})^{s}dA(z)<\infty,$$
 where $\sigma_{a}(z)=\frac{a-z}{1-\overline{a}z}$.
If $p>0$ and $s\geq 2$, then $F(p,p-2,s)$ is the Bloch space. The spaces $F(p,q,s)$ were introduced by Zhao in \cite{zhao} and have attracted attention in recent years.

For an arc $I\subset \partial \mathbb{D}$, let $|I|=\frac{1}{2\pi}\int_{I}|d\zeta|$
be the normalized length of $I$ and
$$S(I)=\{r\zeta\in \mathbb{D}:\ 1-|I|\leq r<1, \zeta\in I\}$$
be the Carleson square in $\mathbb{D}$. Let $0<p<\infty$ and $\mu$ be a positive Borel measure on $\mathbb{D}$. We say that
$\mu$ is a $p$-Carleson measure if
$$\|\mu\|_{p}=\sup_{I\subset \partial \mathbb{D}}\frac{\mu(S(I))}{|I|^{p}}<\infty.$$
The $1$-Carleson measure is the classical Carleson measure.
Moreover, if
$$\lim_{|I|\rightarrow 0}\frac{\mu(S(I))}{|I|^{p}}=0,$$
we say tat $\mu$ is a vanishing $p$-Carleson measure.

  Let $0<p<\infty$, $0<q<\infty$ and $\mu$ be a positive Borel measure on  $\mathbb{D}$.  We say that $f$ belongs to the tent space  $  T_{p,q}(\mu)$ if
$$\|f\|_{T_{p,q}}^{p}=\sup_{I\subset\partial \mathbb{D}}\frac{1}{|I|^{q}}\int_{S(I)}|f(z)|^{p}d\mu(z)<\infty.$$

Embedding theorem of Hardy space was first characterized by   Carleson in \cite{car} and \cite{car1}. The following result is the well known Carleson embedding theorem  for the Hardy space.

 \vskip0.2cm \noindent {\bf Theorem A.} {\it Let $\mu$ be a positive Borel measure on $\mathbb{D}$.
 Then the inclusion  mapping $i:\ H^{2}(\mathbb{D})\rightarrow L^{2}(\mathbb{D},
 d\mu)$ is bounded if and only if $\mu$ is a Carleson measure.}
\vskip0.2cm

Duren generalized the above result to the case of the Bergman space (see \cite{DS}).
 Stegenga studied the Carleson embedding theorem  for Dirichlet type spaces in \cite{SD}.  Wu in \cite{WZ}   proved the following theorem.

  \vskip0.2cm \noindent {\bf Theorem B.} {\it Let $0<p\leq2$ and $\mu$ be a positive Borel measure on $\mathbb{D}$.
 Then the inclusion  mapping $i:\ \mathcal{D}^{p}_{p-1}\rightarrow L^{p}(\mathbb{D},
 d\mu)$ is bounded if and only if $\mu$ is a Carleson measure.}
\vskip0.2cm

In \cite{WZ}, Wu left an open problem about $p>2$ and  conjectured that Theorem B remains true for $2<p<\infty$. In \cite{GP},
 Girela and Pel\'{a}ez proved that this conjecture is false. Bearing in mind of Duren's embedding theorem of Hardy spaces,
  Girela and Pel\'{a}ez \cite{GP1} proved the following result.

   \vskip0.2cm \noindent {\bf Theorem C.} {\it Let $0<p<q<\infty$ and $\mu$ be a positive Borel measure on $\mathbb{D}$.
 Then the inclusion
 mapping $i:\ \mathcal{D}^{p}_{p-1}\rightarrow L^{q}(\mathbb{D},
 d\mu)$ is bounded if and only if $\mu$ is a $\frac{q}{p}$-Carleson measure.}
\vskip0.2cm

 Arcozzi, Rochberg and Sawyer investigated the embedding theorems for Besov spaces in \cite{ARS}.  Xiao in \cite{xiao} studied the embedding theorem  from $\mathcal{Q}_{p}$ spaces to tent spaces. Several embedding theorems for analytic function spaces are also obtain in \cite{LL}, \cite{LLZ} and \cite{WJ}.

   Based on Carleson measure theorem of Bergman spaces, we give a complete characterization of embedding theorem from the   space $\mathcal{D}^{p}_{p-1}$ into tent spaces  $  T_{p,q}(\mu)$. As an application, we study the Volterra integral operator and multipliers for the space $\mathcal{D}^{p}_{p-1}$.

For two quantities $A$ and $B$, the symbol $A\approx B$ means that $A\lesssim B\lesssim A$. We say that $A\lesssim B$ if there exists a constant $C$ such that $A\leq CB$.

\section{embedding theorem}

In this section, we   give a complete characterization for the boundedness and compactness of the inclusion
 mapping $i:\ \mathcal{D}^{p}_{p-1}\rightarrow T_{p,q}(\mu)$.

\begin{theorem}\label{th1}
 Let $0<p<\infty$, $0<q<\infty$ and $\mu$ be a positive Borel measure on $\mathbb{D}$. Then the following statements are equivalent.
  \par (1)~ The inclusion  mapping $i:\ \mathcal{D}^{p}_{p-1}\rightarrow T_{p,q}(\mu)$ is bounded;
\par (2)~  $\mu$ is a $(q+1)$-Carleson measure;
\par (3)~  The inclusion mapping $i:\ A_{q-1}^{p}\rightarrow L^{p}(\mu)$ is bounded.
\end{theorem}

\begin{proof} The $\it (2)\Leftrightarrow (3)$ is well known, see \cite{zhu}. Now we need to prove $\it (1)\Leftrightarrow (2)$.

First we suppose that the inclusion  mapping $i:\ \mathcal{D}^{p}_{p-1}\rightarrow T_{p,q}(\mu)$ is bounded.  For any $I\subset \partial \mathbb{D}$, let $\zeta$ be the midpoint of $I$ and $b=(1-|I|)\zeta$. Set
$$f_{b}(z)=\frac{1-|b|^{2}}{(1-\overline{b}z)^{\frac{p+1}{p}}}.$$
It is easy to see that $f_{b}\in \mathcal{D}^{p}_{p-1}$ with $\|f_{b}\|_{\mathcal{D}^{p}_{p-1}}\lesssim 1$.
Since $|1-\overline{b}z|\thickapprox 1-|b|^{2}\thickapprox|I|,\ \ z\in S(I),$ we get
\begin{equation}\nonumber
\begin{split}
\frac{\mu(S(I))}{|I|^{q+1}}&\thickapprox\frac{1}{|I|^{q}}\int_{S(I)}|f_{b}(z)|^{p}d\mu(z)
 \lesssim
\|f_{b}\|_{T_{p,q}}^{p}<\infty,
\end{split}
\end{equation}
which implies that $\mu$ is $q+1$-Carleson measure.

Conversely, let  $\mu$ be a $q+1$-Carleson measure. Then
 $$\sup_{I\subset \partial \mathbb{D}}\frac{\mu(S(I))}{|I|^{q+1}}<M,$$
 where $M$ is a positive constant.
 For any $I\subset \partial \mathbb{D}$, let $\zeta$ be the midpoint of $I$ and $a=(1-|I|)\zeta$. Note that
\begin{equation}\nonumber
\frac{1}{|I|^{q}}\int_{S(I)}|f(z)|^{p}d\mu(z)\lesssim \frac{1}{|I|^{q}}\int_{S(I)}|f(a)|^{p}d\mu(z)+\frac{1}{|I|^{q}}\int_{S(I)}|f(z)-f(a)|^{p}d\mu(z).
\end{equation}
Since
$|f(z)|\lesssim \frac{\|f\|_{\mathcal{D}^{p}_{p-1}}}{(1-|z|^{2})^{\frac{1}{p}}}, $  we get
$$\frac{1}{|I|^{q}}\int_{S(I)}|f(a)|^{p}d\mu(z)\lesssim \|f\|_{\mathcal{D}^{p}_{p-1}}^{p}.$$

We divide the next proof into two cases, $0<q<1$ and $1\leq q<\infty$.

For $0<q<1$, noting that the inclusion mapping $i:\ A_{q-1}^{p}\rightarrow L^{p}(\mu)$ is bounded, we get

\begin{equation}\nonumber
\begin{split}
&\quad \frac{1}{|I|^{q}}\int_{S(I)}|f(z)-f(a)|^{p}d\mu(z)\\&\thickapprox (1-|a|^{2})^{2-q}\int_{S(I)}\frac{|f(z)-f(a)|^{p}(1-|a|^{2})^{2}}{|1-\overline{a}z|^{4}}d\mu(z)
\\&\lesssim (1-|a|^{2})^{2-q}\int_{\mathbb{D}}\frac{|f(z)-f(a)|^{p}(1-|a|^{2})^{2}}{|1-\overline{a}z|^{4}}d\mu(z)
\\&\lesssim  (1-|a|^{2})^{2-q}\int_{\mathbb{D}}\frac{|f(z)-f(a)|^{p}(1-|a|^{2})^{2}}{|1-\overline{a}z|^{4}}(1-|z|^{2})^{q-1}dA(z)
\\&= (1-|a|^{2})^{2-q}\int_{\mathbb{D}}|f\circ\sigma_{a}(w)-f(a)|^{p}(1-|\sigma_{a}(w)|^{2})^{q-1}dA(w)
\\&\lesssim (1-|a|^{2})\int_{\mathbb{D}}|f\circ\sigma_{a}(w)-f(a)|^{p}(1-|w|^{2})^{q-1}dA(w)
\\&\lesssim (1-|a|^{2})\|f\circ\sigma_{a}(w)-f(a)\|_{\mathcal{D}^{p}_{p-1}}^{p}
\\&\lesssim\int_{\mathbb{D}}|f'(z)|^{p}(1-|z|^{2})^{p-1}dA(z)
\\&\lesssim \|f\|_{\mathcal{D}^{p}_{p-1}}^{p}.
\end{split}
\end{equation}

For $1\leq q<\infty$, we have
\begin{equation}\nonumber
\begin{split}
&\quad \frac{1}{|I|^{q}}\int_{S(I)}|f(z)-f(a)|^{p}d\mu(z)\\&\thickapprox (1-|a|^{2})\int_{S(I)}\frac{|f(z)-f(a)|^{p}(1-|a|^{2})^{2}}{|1-\overline{a}z|^{3+q}}d\mu(z)
\\&\lesssim (1-|a|^{2})\int_{\mathbb{D}}\frac{|f(z)-f(a)|^{p}(1-|a|^{2})^{2}}{|1-\overline{a}z|^{3+q}}d\mu(z)
\\&\lesssim  (1-|a|^{2})\int_{\mathbb{D}}\frac{|f(z)-f(a)|^{p}(1-|a|^{2})^{2}}{|1-\overline{a}z|^{3+q}}(1-|z|^{2})^{q-1}dA(z)
\\&\lesssim  (1-|a|^{2})\int_{\mathbb{D}}\frac{|f(z)-f(a)|^{p}(1-|a|^{2})^{2}}{|1-\overline{a}z|^{4}}dA(z)
\\&\lesssim (1-|a|^{2})\int_{\mathbb{D}}|f\circ\sigma_{a}(w)-f(a)|^{p}dA(w)
\\&\lesssim (1-|a|^{2})\|f\circ\sigma_{a}(w)-f(a)\|_{\mathcal{D}^{p}_{p-1}}^{p}
\\&\lesssim\int_{\mathbb{D}}|f'(z)|^{p}(1-|z|^{2})^{p-1}dA(z)
\\&\lesssim \|f\|_{\mathcal{D}^{p}_{p-1}}^{p}.
\end{split}
\end{equation}
Therefore,
$$\frac{1}{|I|^{q}}\int_{S(I)}|f(z)|^{p}d\mu(z)\lesssim  \|f\|_{\mathcal{D}^{p}_{p-1}}^{p}.$$
So the inclusion
 mapping $i:\ \mathcal{D}^{p}_{p-1}\rightarrow T_{p,q}(\mu)$ is bounded. The proof is complete.
\end{proof}

Similar to the proof of Lemma 4 of \cite{LL}, we have the following lemma.

\begin{lemma}
For $0<r<1$, let $\chi_{\{z:|z|<r\}}$ be the characterization function of the set $\{z:|z|<r\}$. If $\mu$ is a $p$-Carleson measure on $\mathbb{D}$,
then $\mu$ is a vanishing $p$-Carleson measure if and only if $\|\mu-\mu_{r}\|_{p}\rightarrow 0$ as $r\rightarrow 1^{-}$, where $\mu_{r}=\chi_{\{z:|z|<r\}}\mu$.
\end{lemma}

\begin{theorem}
 Let $0<p<\infty$, $0<q<\infty$ and $\mu$ be a positive Borel measure on $\mathbb{D}$. Then the following statements are equivalent.
  \par (1)~ The inclusion
 mapping $i:\ \mathcal{D}^{p}_{p-1}\rightarrow T_{p,q}(\mu)$ is compact;
\par (2)~  $\mu$ is a vanishing $(q+1)$-Carleson measure;
\par (3)~  The inclusion mapping $i:\ A_{q-1}^{p}\rightarrow L^{p}(\mu)$ is compact.
\end{theorem}

\begin{proof}   $\it (2)\Leftrightarrow (3)$ follows by \cite[Theorem 7.8]{zhu}. Now we shall prove that $\it (1)\Leftrightarrow (2)$.

First, we suppose that  $i:\ \mathcal{D}^{p}_{p-1}\rightarrow T_{p,q}(\mu)$ is compact. Let $\{I_{n}\}$ be a sequence of subarcs of $\partial \mathbb{D}$ such that $|I_{n}|\rightarrow 0$ as $n\rightarrow \infty$. Let $w_{n}=(1-|I_{n}|)\zeta_{n}$, where $\zeta_{n}$ is the center of $I_{n}$. Take
$$f_{n}(z)=\frac{1-|w_{n}|^{2}}{(1-\overline{w_{n}}z)^{\frac{p+1}{p}}}.$$
Simple calculation shows that $\|f_{n}\|_{\mathcal{D}^{p}_{p-1}}\lesssim 1$ and the sequence $\{f_{n}\}$ converges to zero uniformly on compact subsets of $\mathbb{D}$.
Then
\begin{equation}\nonumber
\begin{split}
\frac{\mu(S(I_{n}))}{|I_{n}|^{q+1}}&\thickapprox \frac{1}{|I_{n}|^{q}}\int_{S(I_{n})}|f_{n}(z)|^{p}d\mu(z)
 \lesssim\|f_{n}\|_{T_{p,q}}^{p}\rightarrow 0\ \ \ \mbox{as}\ n\rightarrow \infty.
\end{split}
\end{equation}
Hence $\mu$ is a  vanishing $(q+1)$-Carleson measure by the arbitrary of the sequence $I_{n}$.

Conversely, suppose that $\mu$ is a vanishing $(q+1)$-Carleson measure. Then $\mu-\mu_{r}~(0<r<1)$ is also a bounded $(q+1)$-Carleson measure
and $\|\mu-\mu_{r}\|_{p}\rightarrow 0$ as $r\rightarrow 1$. Now, suppose that $\{f_{n}\}\subset \mathcal{D}^{p}_{p-1} $ and $\|f_{n}\|_{\mathcal{D}^{p}_{p-1}}\lesssim 1$
with $f_{n}\rightarrow 0\ (n\rightarrow \infty)$ uniformly on compact subsets of $\mathbb{D}$.
It follows from Theorem $\ref{th1}$ that
\begin{equation}\nonumber
\begin{split}
&\quad\frac{1}{|I_{n}|^{q}}\int_{S(I_{n})}|f_{n}(z)|^{p}d\mu(z)
\\&\lesssim\frac{1}{|I_{n}|^{q}}\int_{S(I_{n})}|f_{n}(z)|^{p}d\mu_{r}(z)+\frac{1}{|I_{n}|^{q}}\int_{S(I_{n})}|f_{n}(z)|^{p}d(\mu-\mu_{r})(z)
\\&\lesssim \frac{1}{|I_{n}|^{q}}\int_{S(I_{n})}|f_{n}(z)|^{p}d\mu_{r}(z)+\|\mu-\mu_{r}\|_{q+1}\|f_{n}(z)\|_{\mathcal{D}^{p}_{p-1}}^{p}
\\&\lesssim \frac{1}{|I_{n}|^{q}}\int_{S(I_{n})}|f_{n}(z)|^{p}d\mu_{r}(z)+\|\mu-\mu_{r}\|_{q+1}.
\end{split}
\end{equation}
Letting $r\rightarrow 1$ and $n\rightarrow \infty$, we have
$$\lim_{n\rightarrow \infty}\|f_{n}\|_{T_{p,q}}=0.$$
So the  inclusion
 mapping $i:\ \mathcal{D}^{p}_{p-1}\rightarrow T_{p,q}(\mu)$ is compact. The proof is complete.
\end{proof}\bsk

\section{Norm of integral operators}

 In this section, we consider the application of our embedding theorem to Volterra type operators.  Recall that, for $g\in H(\D)$, the Volterra type operator, denoted by $T_g$, is defined by (see, e.g., \cite{p,sz})
$$T_gf(z)=\int_0^z f(\xi)g'(\xi)d\xi , ~~~~f\in H(\D), \qquad z\in \mathbb{D}. $$
 Similarly, another integral operator was defined by
$$
 I_gf(z)= \int_0^z f'(\xi)g(\xi)d\xi, ~~~~f\in H(\D), \qquad z\in \mathbb{D}.
$$
  The  operators $T_g$ and $I_g$ are important operators since they are closely related with multiplication operator, i.e.,
  $$ T_g f+I_g f+f(0)g(0)=M_g f,$$
where $M_g$ is the multiplication operator which defined by $M_gf(z)=f(z)g(z). $ In \cite{p}, Pommerenke  showed that $T_g$ is
 bounded  on  $H^2$ if
and only if $g\in$ $BMOA$. Aleman and Siskakis studied
$T_g$ on $H^p$ and weighted Bergman spaces in
\cite{ac, as2}.   Recently, the operators $T_g$ and $I_g$ between some spaces of analytic
functions were investigated in \cite{ac, as2, cls, dlz, LLL,   lsjia,   ls3,    lsmn,  PZ, ql, sl,  sz} (see also the related references
therein).

\begin{theorem}\label{th2} Let $0<p<\infty$ and $g\in H(\D)$.  Then $T_{g}:\ \mathcal{D}^{p}_{p-1}\rightarrow F(p,p-1,1)$ is bounded if and only if $g\in \mathcal{B}$. Moreover
$$\|T_{g}\|_{\mathcal{D}^{p}_{p-1}\rightarrow F(p,p-1,1) }\approx \|g\|_{\mathcal{B}}.$$
\end{theorem}

\begin{proof} Suppose that $T_{g}:\ \mathcal{D}^{p}_{p-1}\rightarrow F(p,p-1,1)$ is bounded.
Set $$f_{a}(z)=\frac{1-|a|^{2}}{(1-\overline{a}z)^{\frac{p+1}{p}}} ,\ \ a\in \mathbb{D}.$$
Simple computation shows that $\|f_{a}\|_{\mathcal{D}^{p}_{p-1}}\lesssim 1$.
Then by Lemma 4.12 of \cite{zhu}, we get
 \begin{equation}\nonumber
\begin{split}
\infty&>\|T_{g}\|^{p}\|f_{a}\|_{\mathcal{D}^{p}_{p-1}}^{p}\\&\gtrsim\|T_{g}f_{a}\|_{F(p,p-1,1)}^{p}\\&=\int_{\mathbb{D}}|f_{a}(z)|^{p}|g'(z)|^{p}(1-|z|^{2})^{p-1}(1-|\sigma_{a}(z)|^{2})dA(z)\\&
=\int_{\mathbb{D}}|g'(z)|^{p}\frac{(1-|a|^{2})^{p}}{|1-\overline{a}z|^{p+1}}(1-|z|^{2})^{p-1}(1-|\sigma_{a}(z)|^{2})dA(z)\\&
=\int_{\mathbb{D}}|g'(\sigma_{a}(w))|^{p}\frac{(1-|a|^{2})^{p}}{|1-\overline{a}\sigma_{a}(w)|^{p+1}}(1-|\sigma_{a}(w)|^{2})^{p-1}(1-|w|^{2})|\sigma_{a}'(w)|^{2}dA(z)
\\&\gtrsim (1-|a|^{2})^{p}|g'(a)|^{p}.
\end{split}
\end{equation}
Since $a$ is arbitrary in $\mathbb{D}$, then $g\in \mathcal{B}$ and $ \|g\|_{\mathcal{B}}\lesssim \|T_{g}\|_{\mathcal{D}^{p}_{p-1}\rightarrow F(p,p-1,1) }$.

Conversely,  suppose that $g\in \mathcal{B}$. For any $I\subset \partial \mathbb{D}$, let $\zeta$ be the midpoint of $I$ and $a=(1-|I|)\zeta$. We have
$$|1-\overline{a}z|\approx 1-|a|^{2}\approx |I|, \ \ z\in S(I).$$
  Let $d\mu(z)=:|g'(z)|^{p}(1-|z|^{2})^{p}dA(z)$. Then $\mu$ is $2$-Carleson measure and $\|\mu\|\approx\|g\|_{\mathcal{B}}^{p}$ since $ \mathcal{B}=F(p,p-2,2)$. Applying the embedding theorem, we get
\begin{equation}\nonumber
\begin{split}
&\quad \int_{\mathbb{D}}|f(z)|^{p}|g'(z)|^{p}(1-|z|^{2})^{p-1}(1-|\sigma_{a}(z)|^{2})dA(z)
\\&\approx \frac{1}{|I|}\int_{S(I)}|f(z)|^{p}|g'(z)|^{p}(1-|z|^{2})^{p}dA(z)
\lesssim\|g\|_{\mathcal{B}}^{p}\|f\|_{\mathcal{D}^{p}_{p-1}}^{p}.
\end{split}
\end{equation}
It follows that $$\|T_{g}f\|_{F(p,p-1,1)}^{p}\lesssim \|g\|_{\mathcal{B}}^{p}\|f\|_{\mathcal{D}^{p}_{p-1}}^{p}.$$
So
$T_{g}:\ \mathcal{D}^{p}_{p-1}\rightarrow F(p,p-1,1)$ is bounded and $\|T_{g}\|_{\mathcal{D}^{p}_{p-1}\rightarrow F(p,p-1,1) }  \lesssim \|g\|_{\mathcal{B}}$.
 The proof is complete. \end{proof}

\begin{theorem} Let $0<p<\infty$ and $g\in H(\D)$. Then $I_{g}:\ \mathcal{D}^{p}_{p-1}\rightarrow F(p,p-1,1)$ is bounded if and only if $g\in H^{\infty}$. Moreover $$\|I_{g}\|_{\mathcal{D}^{p}_{p-1}\rightarrow F(p,p-1,1) }\approx \|g\|_{\infty}.$$
\end{theorem}

\begin{proof} First, we assume that $g\in H^{\infty}$. We obtain that
\begin{equation}\nonumber
\begin{split}
\frac{1}{|I|}\int_{S(I)}|f'(z)|^{p}|g(z)|^{p}(1-|z|^{2})^{p}dA(z)&\lesssim\|g\|_{\infty}^{p}\|f\|_{F(p,p-1,1)}^{p}
\\&\lesssim \|g\|_{\infty}^{p}
\|f\|_{\mathcal{D}^{p}_{p-1}}^{p}.
\end{split}
\end{equation}
Then
$$ \|I_{g}f\|_{F(p,p-1,1)}\lesssim \|g\|_{\infty}
\|f\|_{\mathcal{D}^{p}_{p-1}}.$$
This implies that $I_{g}$ is bounded form $\mathcal{D}^{p}_{p-1}$ to $ F(p,p-1,1)$, and $\|I_{g}\|_{\mathcal{D}^{p}_{p-1}\rightarrow F(p,p-1,1) }\gtrsim \|g\|_{\infty}$.

Conversely,  suppose that $I_{g}:\ \mathcal{D}^{p}_{p-1}\rightarrow F(p,p-1,1)$ is bounded. We choose the test function
$$f_{a}(z)=\frac{1}{\overline{a}}\frac{1-|a|^{2}}{(1-\overline{a}z)^{\frac{p+1}{p}}}, ~\mbox{with}~  a\in \mathbb{D}~\mbox{and}~ a\neq 0.$$
Then
 $f_{a}\in \mathcal{D}^{p}_{p-1}$ with $\|f_{a}\|\lesssim 1$.
So $$\|I_{g}f_{a}\|_{F(p,p-1,1)}\lesssim \|I_{g}\|<\infty.$$
Note that
$$|g(a)|^{p}\lesssim \int_{\mathbb{D}}\frac{1}{|1-\overline{a}z|}|g\circ\sigma_{a}(z)|^{p}(1-|z|^{2})^{p}dA(z)\lesssim \|I_{g}f_{a}\|_{F(p,p-1,1)}^{p}.$$
We have $g\in H^{\infty}$ and $\|g\|_{\infty}\lesssim \|I_{g}\|_{\mathcal{D}^{p}_{p-1}\rightarrow F(p,p-1,1) } $. The proof is complete.
\end{proof}

\begin{theorem} Let $0<p<\infty$ and $g\in H(\D)$. Then  $M_{g}:\ \mathcal{D}^{p}_{p-1}\rightarrow F(p,p-1,1)$ is bounded if and only if $g\in H^{\infty}$. Moreover, $$\|M_{g}\|_{\mathcal{D}^{p}_{p-1}\rightarrow F(p,p-1,1) } \approx \|g\|_{\infty}.$$
\end{theorem}

\begin{proof}
Note that $H^{\infty}\subset \mathcal{B}$. If $g\in H^{\infty}$. Then both $I_{g}$ and $T_{g}$ are bounded from $\mathcal{D}^{p}_{p-1}$ to $ F(p,p-1,1)$. Then
$M_{g}:\ \mathcal{D}^{p}_{p-1}\rightarrow F(p,p-1,1)$ is bounded and $\|M_{g}\|_{\mathcal{D}^{p}_{p-1}\rightarrow F(p,p-1,1) }\lesssim \|g\|_{\infty}$.

On the other hand, suppose that  $M_{g}:\ \mathcal{D}^{p}_{p-1}\rightarrow F(p,p-1,1)$ is bounded.
Set $f_{w}(z)=\frac{1-|w|^{2}}{(1-\overline{w}z)^{1+\frac{1}{p}}}.$ Then $f_{w}\in \mathcal{D}^{p}_{p-1}$ with
$\|f_{w}\|_{\mathcal{D}^{p}_{p-1}}\lesssim 1$.
Using Lemma 4.12 of \cite{zhu}, we obtain that
\begin{equation}\nonumber
\begin{split}
&\quad\int_{\mathbb{D}}|f'(z)|^{p}(1-|z|^{2})^{p-1}(1-|\sigma_{a}(z)|^{2})dA(z)\\&
=\int_{\mathbb{D}}|f'(\sigma_{a}(w))|^{p}(1-|\sigma_{a}(w)|^{2})^{p-1}(1-|w|^{2})|\sigma_{a}'(w)|^{2}dA(w)
\\&\gtrsim |f'(a)|(1-|a|^{2})^{1+\frac{1}{p}}.
\end{split}
\end{equation}
Similarly with the proof of Lemma 2.5 of \cite{LLL}, we have
$$|f(z)|\lesssim \frac{\|f\|_{F(p,p-1,1)}}{(1-|z|^{2})^{\frac{1}{p}}}.$$
  Then
\begin{equation}\nonumber
\begin{split}\Big|\frac{1-|w|^{2}}{(1-\overline{w}z)^{1+\frac{1}{p}}}g(z)\Big|&\lesssim \frac{1}{(1-|z|^{2})^{\frac{1}{p}}}\|M_{g}f_{w}\|_{F(p,p-1,1)}
\\&\lesssim \frac{1}{(1-|z|^{2})^{\frac{1}{p}}}\|M_{g}\|_{\mathcal{D}^{p}_{p-1}\rightarrow F(p,p-1,1) }.
\end{split}
\end{equation}
Let $z=w$. Then $\|g\|_{\infty}\lesssim \|M_{g}\|_{\mathcal{D}^{p}_{p-1}\rightarrow F(p,p-1,1) }$ since $z$ was arbitrary. The proof is complete.
\end{proof}\bsk

\section{Essential norm of integral operator}

In this section, we investigate the essential norm of integral operators from $\mathcal{D}^{p}_{p-1}$ to $F(p,p-1,1)$. Recall that the essential norm of $T: X\rightarrow Y$ is defined as follows.
$$
\|T\|_{e, X\rightarrow Y}=\inf\{\|T-K\|_{X\rightarrow Y}: K~\mbox{is compact}~~\}.
$$
Here $X$ and $Y$ are Banach spaces and $T: X\rightarrow Y$ is a bounded linear operator.  It is well known that $\|T\|_{e, X\rightarrow Y}= 0$ if and only if $T: X\rightarrow Y$ is compact.

Given two Banach spaces $X$ and $Y$ with $Y\subset X$. For $f\in X$, the distance of function $f$ to the space $Y$ is defined by
$${\dist}_X(f, Y)=\inf_{g\in Y}\|f-g\|_{X}.$$

The following lemma, which characterized the distance from the Bloch function to the little Bloch space£¬ was proved by  Attele \cite{Ak} and Tjani \cite{TM}.
Here and afterward, we denote $g_{r}(z)=g(rz)$ with $0<r<1$.

\begin{lemma}\label{lem2}
If $g\in \mathcal{B}$, then
$$\limsup_{|z|\rightarrow 1^{-}}(1-|z|^{2})|g'(z)|\leq {\dist}_{\mathcal{B}}(g,\mathcal{B}_{0})\leq \limsup_{r\rightarrow 1^{-}}\|g-g_{r}\|_{\mathcal{B}}\leq 2\limsup_{|z|\rightarrow 1^{-}}(1-|z|^{2})|g'(z)|~. $$
\end{lemma}

For the proof of the next theorem, we need the following lemma.

\begin{lemma} Let $0<p<\infty$, $g\in \mathcal{B}$ and $0<r<1$. Then $T_{g_{r}} $ is a compact operator from $\mathcal{D}^{p}_{p-1}$ to $F(p,p-1,1)$.
\end{lemma}

\begin{proof}
Let $\{f_{n}\}$ be a sequence such that $\|f_{n}\|_{ \mathcal{D}^{p}_{p-1}}\leq 1$ and $f_{n}\rightarrow 0$ uniformly on compact subsets of
$\mathbb{D}$ as $n\rightarrow \infty$. We only need to show that
$$\lim_{n\rightarrow \infty}\|T_{g_{r}}f_{n}\|_{F(p,p-1,1)}=0.$$
From the proof of Lemma 2, we have  $\|g_{r}\|_{\mathcal{B}}\leq \|g\|_{\mathcal{B}}$.  Note that
$$|g'_{r}(z)|\leq \frac{(1-r^{2}|z|^{2})|g'(rz)|}{1-r^{2}}\leq \frac{\|g\|_{\mathcal{B}}}{1-r^{2}},~ z\in \mathbb{D},$$
we have
\begin{equation}\nonumber
\begin{split}
\|T_{g_{r}}f_{n}\|_{F(p,p-1,1)}^{p}& =\sup_{a\in \mathbb{D}}\int_{\mathbb{D}}|f_{n}(z)|^{p}|g_{r}'(z)|^{p}(1-|z|^{2})^{p-1}(1-|\sigma_{a}(z)|^{2})dA(z)
\\&\lesssim  \frac{\|g\|_{\mathcal{B}}^{p}}{(1-r^{2})^{p}}\sup_{a\in \mathbb{D}}\int_{\mathbb{D}}|f_{n}(z)|^{p}(1-|z|^{2})^{p-1}dA(z).
\end{split}
\end{equation}
Note that
$\|f\|_{\mathcal{D}^{p}_{p-1}}\leq 1$ and then $|f_{n}(z)|^{p}(1-|z|^{2})^{p-1}\lesssim \|f\|_{\mathcal{D}^{p}_{p-1}}\leq  1$.
The desired result follows from the dominated convergence theorem.
\end{proof}

Next, we give some estimates of the essential norm of $T_{g}$ from $\mathcal{D}^{p}_{p-1}$ to $ F(p,p-1,1)$.

\begin{theorem} Let $0<p<\infty$ and $g\in H(\D)$ such that $T_{g}:~\mathcal{D}^{p}_{p-1}\rightarrow F(p,p-1,1)$  is bounded.   Then
$$\|T_{g}\|_{e, \mathcal{D}^{p}_{p-1}\rightarrow F(p,p-1,1) }\approx \limsup_{|z|\rightarrow 1^{-}}(1-|z|^{2})|g'(z)|\approx {\dist}_{\mathcal{B}}(g,\mathcal{B}_{0}).$$
\end{theorem}

\begin{proof} Let $a_{n}\in \mathbb{D}$ such that $|a_{n}|\rightarrow 1$ as $n\rightarrow \infty$. Set
$$f_{n}(z)=\frac{1-|a_{n}|^{2}}{(1-\overline{a_{n}}z)^{\frac{p+1}{p}}}, ~~~~z\in \D.$$
It is easy to see that $\|f_{n}\|_{\mathcal{D}^{p}_{p-1}}\lesssim 1$ and the sequence $\{f_{n}\}_{n=1}^{\infty}$ converges to zero uniformly on the compact subsets of $\mathbb{D}$. For any compact operator $K:~\mathcal{D}^{p}_{p-1}\rightarrow F(p,p-1,1)$, we have
$$\lim_{n\rightarrow \infty}\|Kf_{n}\|_{ F(p,p-1,1)}=0.$$
Therefore,
\begin{equation}\nonumber
\begin{split}
&\quad\|T_{g}-K\|_{\mathcal{D}^{p}_{p-1}\rightarrow F(p,p-1,1) }\\&\gtrsim \limsup_{n\rightarrow \infty}\|(T_{g}-K)(f_{n})\|_{F(p,p-1,1)}\\&
\geq\limsup_{n\rightarrow \infty}(\|T_{g}f_{n}\|_{F(p,p-1,1)}-\|Kf_{n}\|_{ F(p,p-1,1)})\\&
=\limsup_{n\rightarrow \infty}\|T_{g}f_{n}\|_{F(p,p-1,1)}\\&=
\limsup_{n\rightarrow \infty}\Big(\int_{\mathbb{D}}|f_{n}(z)|^{p}|g'(z)|^{p}(1-|z|^{2})^{p-1}(1-|\sigma_{a_{n}}(z)|^{2})dA(z)\Big)^{1/p}
\\&\gtrsim \limsup_{n\rightarrow \infty}(1-|a_{n}|^{2})|g'(a_{n})|.
\end{split}
\end{equation}
Since the sequence ${a_{n}}$ is arbitrary in $\mathbb{D}$, we get that
$$\|T_{g}\|_{e, \mathcal{D}^{p}_{p-1}\rightarrow F(p,p-1,1) }\gtrsim \limsup_{|z|\rightarrow 1^{-}}(1-|z|^{2})|g'(z)|.$$

On the other hand, $T_{g_{r}}:~\mathcal{D}^{p}_{p-1}\rightarrow F(p,p-1,1)$ is a compact operator. Combining this with Theorem $\ref{th2}$, we obtain that
$$\|T_{g}\|_{e, \mathcal{D}^{p}_{p-1}\rightarrow F(p,p-1,1) }\leq \|T_{g}-T_{g_{r}}\|_{ \mathcal{D}^{p}_{p-1}\rightarrow F(p,p-1,1) }=\|T_{g-g_{r}}\|_{ \mathcal{D}^{p}_{p-1}\rightarrow F(p,p-1,1) }$$
$$\approx \|g-g_{r}\|_{\mathcal{B}}.$$
Then
$$\|T_{g}\|_{e, \mathcal{D}^{p}_{p-1}\rightarrow F(p,p-1,1) } \lesssim \limsup_{r\rightarrow 1^{-}}\|g-g_{r}\|_{\mathcal{B}}\approx {\dist}_\mathcal{B}(g,\mathcal{B}_{0}).$$
This completes the proof by Lemma $\ref{lem2}$ .
\end{proof}

From the last Theorem, we can easily get the following corollary.

\begin{corollary} Let $0<p<\infty$ and  $g\in \mathcal{B}$.  Then $T_{g}:~\mathcal{D}^{p}_{p-1}\rightarrow F(p,p-1,1)$ is compact if and only if $g\in \B_0$.
\end{corollary}

\end{document}